\documentclass{article}
\usepackage[T1]{fontenc}
\usepackage{amsthm, fullpage}
\usepackage{amsmath}
\usepackage{amssymb, mathrsfs}
\usepackage{amsfonts}
\usepackage{eucal}
\usepackage[all]{xy}
\usepackage{pslatex}
\usepackage{hyperref}

\newtheorem{prop}{Proposition}[section]
\newtheorem{theo}[prop]{Theorem}
\newtheorem{coro}[prop]{Corollary}
\newtheorem{lemm}[prop]{Lemma}
\newtheorem{lem}[prop]{Lemma}
\newtheorem*{lemm*}{Lemma}

\theoremstyle{definition}

\newtheorem{empt}[prop]{}
\newtheorem{dfn}[prop]{Definition}
\newtheorem{rem}[prop]{Remark} 
 
\newtheorem*{nota*}{Notation}
\newtheorem*{rem*}{Remark}

\theoremstyle{thm}
\newtheorem{thm}[prop]{Theorem}
\newtheorem*{thm*}{Theorem}
\newtheorem*{lem*}{Lemma}

\newtheorem*{cor*}{Corollary}
\newtheorem*{prop*}{Proposition}

\theoremstyle{dfn}
\newtheorem*{dfn*}{Definition}

\numberwithin{equation}{prop}

\newcommand{\riso}{ \overset{\sim}{\longrightarrow}\, }
\newcommand{\liso}{ \overset{\sim}{\longleftarrow}\, }

\newcommand{\Spec}{\mathrm{Spec}\,}

\newcommand{\FF}{{\mathcal{F}}}

\newcommand{\E}{{\mathcal{E}}}
\newcommand{\G}{{\mathcal{G}}}
\renewcommand{\H}{{\mathcal{H}}}

\newcommand{\D}{{\mathcal{D}}}

\newcommand{\K}{{\mathcal{K}}}

\renewcommand{\O}{{\mathcal{O}}}

\newcommand{\V}{\mathcal{V}}

\newcommand{\A}{\mathbb{A}}
\renewcommand{\P}{\mathbb{P}}

\newcommand{\DD}{\mathbb{D}}

\newcommand{\Q}{\mathbb{Q}}
\newcommand{\Z}{\mathbb{Z}}

\begin{document}

\title{Hard Lefschetz Theorem in $p$-adic cohomology}
\author{Daniel Caro} 

\date{}

\maketitle

\begin{abstract}
In this paper, we give a $p$-adic analogue of 
the Hard Leftschetz Theorem. 
\end{abstract}

\date
\tableofcontents

\bigskip

\section*{Introduction}
The main purpose of this paper is to check a $p$-adic analogue of 
the Hard Leftschetz Theorem. 
We have followed the proof in the $l$-adic context written in 
\cite[IV.4.1]{KW-Weilconjecture} (we can compare with the proof given in \cite{BBD}). 
Two main ingredients of the proof are the semi-simplicity of a pure arithmetic $\D$-module (see \cite[4.3.1]{Abe-Caro-weights})
and the construction and the properties of the trace map given in \cite[1.5]{Abe-LanglandsIsoc}.
Then, this paper can be considered as a natural continuation of these works.
We follow here their terminology and notation.

Let us describe the contents of the paper.
In the first  chapter, we study 
the properties of the Serre subcategory consisting of relative constant objects.
In the second chapter, we introduce the $p$-adic analogue of 
the Brylinsky-Radon transform and use its properties to prove the Hard Leftschetz Theorem. 
We have tried to write the proofs only when the $p$-adic analogues were not straightforward. 
Finally, in the last chapter, for the sake of completeness, 
we check 
the inversion formula satisfied by Radon transform. 

\bigskip

In this paper, we fix a complete discrete valuation ring $\V$ of mixed characteristic $(0,p)$. 
Its residue field is denoted by $k$, and assume it to be perfect. We also
suppose that there exists a lifting $\sigma\colon \V\xrightarrow{\sim}\V$ of the $s$-th
Frobenius automorphism of $k$. We put  $q:=p^s$, $K:=\mathrm{Frac}(\V)$.
We fix an isomorphism $\iota\colon\overline{\mathbb{Q}}_p\cong\mathbb{C}$. 

\subsection*{Acknowledgment}
The author would like to thank Tomoyuki Abe for 
many explanations concerning the trace maps and their properties
and Weizhe Zheng for his interest and questions concerning this $p$-adic analogue.

\begin{nota*}
We will keep the notation concerning cohomological operators as in \cite[1.1]{Abe-Caro-weights}.
We will also use the categories defined in \cite[1.5]{Abe-Caro-BEq}.
We recall in this paragraph some of the construction.
Let $X$ be a realizable variety.
Let $\mathrm{Hol}_F(X/K)'$ be the subset of $\mathrm{Ob}(\mathrm{Ovhol}(X/K))$ which
can be endowed with some $s'$-th Frobenius structure for some integer $s'$
which is a multiple of $s$, and let $\mathrm{Hol}_F(X/K)$ be the thick
abelian subcategory generated by $\mathrm{Hol}_F(X/K)'$ in
$\mathrm{Ovhol}(X/K)$. We denote by
$D^{\mathrm{b}}_{\mathrm{hol},F}(X/K)$ the triangulated full subcategory of
$D^{\mathrm{b}}_{\mathrm{ovhol}}(X/K)$ such that the cohomologies are in
$\mathrm{Hol}_F(X)$. 
For any integer $n$, we can extend the twist of Tate over $D^{\mathrm{b}}_{\mathrm{hol},F}(X/K)$:
by definition the twist $(n)$ is the identity (and then the forgetful functor $F\text{-}D^{\mathrm{b}}_{\mathrm{hol},F}(X/K)
\to D^{\mathrm{b}}_{\mathrm{hol},F}(X/K)$ commutes with the twist of Tate). 
For simplicity and if there is no risk of confusion with the notion of holonomicity of Berthelot, 
we will write $D^{\mathrm{b}}_{\mathrm{hol}}(X/K)$ instead of 
$D^{\mathrm{b}}_{\mathrm{hol},F}(X/K)$
and 
$\mathrm{Hol}(X/K)$
instead of 
$\mathrm{Hol}_F(X/K)$.
With this notation, we get 
$F\text{-}D^{\mathrm{b}}_{\mathrm{ovhol}}(X/K)= F\text{-}D^{\mathrm{b}}_{\mathrm{hol}}(X/K)$.
Be careful that this notation is a bit misleading since in general we do not know even with Frobenius structures if the notion of holonomicity of Berthelot
and the notion of overholonomicity coincides. 

\end{nota*}

\section{Constant objects with respect to smooth $\P ^{d}$-fibration morphisms}

\begin{empt}
Let $g\colon U \to T$ be a morphism of realizable varieties. 
Let $\FF,\,\G \in (F \text{-})D ^\mathrm{b} _\mathrm{hol}  (T/K)$.
We have the morphisms 
\begin{equation}
\label{bottom-II.11firstdiag}
\epsilon _g \colon g ^{!} (\FF) \otimes  g ^{+} (\G)
\underset{\mathrm{adj}}{\longrightarrow}
g ^{!} g _{!} (g ^{!} (\FF) \otimes  g ^{+} (\G))
\underset{\mathrm{proj}}{\riso}
g ^{!} (g _{!} g ^{!} (\FF) \otimes  \G)
\underset{\mathrm{adj}}{\longrightarrow}
g ^{!} (\FF \otimes  \G)
\end{equation}
where $\mathrm{proj}$ (resp. $\mathrm{adj}$) means the projection isomorphism constructed in \cite[A.6]{Abe-Caro-weights}
(resp. the adjunction isomorphism corresponding to the adjoint functors $(g _! , g ^!)$).
Since the projection isomorphisms and adjunction isomorphisms are transitive, then so is for 
$\epsilon _g$ i.e., for any $h \colon V \to T$ morphism of realizable varieties, the diagram
\begin{equation}
\label{trans-epsilon}
\xymatrix{
{h ^{!} g ^{!} (\FF) \otimes  h ^{+} g ^{+} (\G)} 
\ar[d] ^-{\sim}
\ar[r] ^-{\epsilon _h}
& 
{h ^{!} ( g ^{!} (\FF) \otimes  g ^{+} (\G)) } 
\ar[r] ^-{h ^{!} \epsilon _g}
& 
{h ^{!} g ^{!} (\FF \otimes  \G)} 
\ar[d] ^-{\sim}
\\ 
{(g \circ h) ^{!} (\FF) \otimes (g \circ h)  ^{+} (\G)} 
\ar[rr] ^-{\epsilon _{g \circ h}}
&& 
{(g \circ h) ^{!} (\FF \otimes \G)} 
}
\end{equation}
is commutative.

\end{empt}

\begin{empt}
[Poincaré duality]
\label{notaPoincaréduality}
Let $f\colon X \to S$ be a smooth equidimensional morphism of relative dimension $d$ of realizable varieties. 
T. Abe has checked 	(see \cite[1.5.13]{Abe-LanglandsIsoc}) that the morphism 
\begin{equation}
\label{dual-smooth-morph}
\theta _{f}\colon f ^{+} [d]  \to f ^{!}[-d] (-d),
\end{equation}
which is induced by adjunction
from the trace map $\mathrm{Tr}\colon f _! f ^{+} [2d] (d) \to id$,
is an isomorphism of t-exact functors. 
This isomorphism satisfies several compatibility properties (see  \cite[1.5]{Abe-LanglandsIsoc}), e.g. it is transitive. 
\end{empt}

\begin{empt}
We keep the notation of \ref{notaPoincaréduality}.
Let $\FF,\,\G \in (F \text{-})D ^\mathrm{b} _\mathrm{hol}  (S/K)$.
The diagram below 
\begin{equation}
\notag
\xymatrix{
*+[F]{f ^{+} (\FF \otimes \G)[2d] (d) } 
\ar[r] ^-{\sim}
\ar[d] ^-{\mathrm{adj}}
& 
*+[F]{f ^{+} (\FF[2d] (d) ) \otimes f ^{+} (\G)} 
\ar[r] ^-{\mathrm{adj}}
\ar[d] ^-{\mathrm{adj}}
& 
{f ^{!} f _{!} f ^{+} (\FF[2d] (d) )\otimes  f ^{+} (\G) } 
\ar[r] ^-{\mathrm{Tr} \otimes Id}
\ar[d] ^-{\mathrm{adj}}
& 
*+[F]{f ^{!} (\FF) \otimes  f ^{+} (\G)} 
\ar[ddd] ^-{\epsilon _f}
\\ 
{f ^{!} f _{!} f ^{+} (\FF \otimes \G)[2d] (d) } 
\ar[r] ^-{\sim}
\ar[dd] ^-{\mathrm{Tr}}
& 
{f ^{!} f _{!}  ( f ^{+} (\FF [2d] (d) ) \otimes f ^{+} (\G) )} 
\ar[r] ^-{\mathrm{adj}}
\ar[d] ^-{\sim} _{\mathrm{proj}}
& 
{f ^{!} f _{!}  (f ^{!} f _{!} f ^{+} (\FF [2d] (d) )\otimes  f ^{+} (\G))} 
\ar[d] ^-{\sim} _{\mathrm{proj}}
& 
{} 
\\
{} 
& 
{f ^{!} ( f _{!} f ^{+} (\FF[2d] (d)) \otimes \G ) } 
\ar[r] ^-{\mathrm{adj}}
\ar@{=}[rd] ^-{}
& 
{f ^{!} (f _{!}  f ^{!} f _{!} f ^{+} (\FF [2d] (d)) \otimes  \G) } 
\ar[d] ^-{\mathrm{adj}}
& 
{}
\\
*+[F]{f ^{!} (\FF \otimes \G)} 
& 
{} 
& 
{f ^{!} ( f _{!} f ^{+} (\FF [2d] (d)) \otimes \G )} 
\ar[ll] ^-{\mathrm{Tr} \otimes Id}
\ar[r] ^-{\mathrm{Tr} \otimes Id}
& 
{f ^{!} (\FF \otimes \G)} 
}
\end{equation}
is commutative. 
Indeed, the pentagon  is commutative from \cite[1.5.1.Var5]{Abe-LanglandsIsoc}.
The other parts of the diagram are commutative by definition and functoriality. 
Hence we get the canonical commutative square: 
\begin{equation}
\label{II.11firstdiag}
\xymatrix{
{f ^{+} (\FF \otimes \G)[2d] (d) } 
\ar[r] ^-{\sim}
\ar[d] ^-{\sim} _-{\theta _f}
& 
{f ^{+} (\FF[2d] (d)) \otimes f ^{+} (\G)} 
\ar[d] ^-{\sim} _-{\theta _f \otimes id}
\\ 
{f ^{!} (\FF \otimes \G)} 
& 
{f ^{!} (\FF) \otimes  f ^{+} (\G).} 
\ar[l] ^-{\epsilon _f}
}
\end{equation}
This implies that the bottom morphism of \ref{II.11firstdiag} is also an isomorphism. 
\end{empt}

\begin{dfn}
Let $f\colon X\to S$ be an equidimensional smooth morphism of relative dimension $d$ of realizable varieties. 
\begin{enumerate}
\item The objects of the essential image of the functor
$f ^{+} \colon 
(F \text{-})D ^\mathrm{b} _\mathrm{hol}  (S/K)
\to 
(F \text{-})D ^\mathrm{b} _\mathrm{hol}  (X/K)$
are called constant (with respect to $f$). 

\item The objects of the essential image of the functor
$f ^{+}[d]
\colon 
(F \text{-})\mathrm{Hol}   (S/K)
\to 
(F \text{-})\mathrm{Hol}  (X/K)$
are called constant (with respect to $f$). 
We denote by 
$f ^{+}[d](F \text{-})\mathrm{Hol}   (S/K)$ its essential image. 
\end{enumerate}

\end{dfn}

\begin{empt}
\label{def-K_X}
Let $X$ be a realizable $k$-variety and
$ p _X \colon X \to \Spec k $ be the structural morphism. 
We denote by $ K _X : = p _{X} ^{+} (K)$ the constant coefficient of $X$. 
The complex $K _X$ is the $p$-adic analogue of the constant sheaf $\Q _{l}$ over $X$.
Let $\E \in D ^\mathrm{b} _\mathrm{hol}  (X/K)$.
We notice that 
$K _{X} \otimes \E \riso \E$.

\end{empt}

\begin{prop}
\label{KW-II.11.2}
Let $u \colon Y \hookrightarrow X$ be a closed immersion of pure codimension $r$ in $X$ of smooth realizable $k$-varieties.
Let $\E \in (F \text{-})D ^\mathrm{b} _\mathrm{hol}  (X/K)$.
\begin{enumerate}
\item There exists a natural functorial morphism of $(F \text{-})D ^\mathrm{b} _\mathrm{hol}  (Y/K)$ of the form
\begin{equation}
\label{KW-II.11.2-part1}
\theta _u \colon u ^{+} (\E) \to u ^{!} (\E) [2r](r). 
\end{equation}
\item 
\label{KW-II.11.2.2}
If (locally on $X$) the complex $\E$ is constant with respect to a smooth equidimensional morphism
$f \colon X \to S$ of realizable varieties
such that $f \circ u$ is also smooth,
then $\theta _u$ is an isomorphism. 
\end{enumerate}
 \end{prop}
 
 \begin{proof}
This can be checked as in \cite[II.11.2]{KW-Weilconjecture}: 
with the notation and hypothesis of the second part, putting $g:= f \circ u$ and $d _g := \dim Y - \dim S$,
for any $\K \in (F \text{-})D ^\mathrm{b} _\mathrm{hol}  (S/K)$, 
by using the isomorphism \ref{dual-smooth-morph}, we get the isomorphism
\begin{equation}
\label{predef1-theta_u}
\theta _u  \colon u ^{+} ( f ^{+}\K) 
\underset{\theta _g}{\riso}
u ^{!} ( f ^{!}\K)  [-2 d _g](d _g)
\underset{u ^{!}(\theta _f ^{-1} )}{\riso}
u ^{!}  ( f ^{+}\K) [2r](r).
\end{equation}
In particular, we get 
$\theta _u \colon u ^{+} (K _X) \riso u ^{!} (K _X) [2r](r)$.
We remark that 
$\theta _u \colon u ^{+} (K _X) \riso u ^{!} (K _X) [2r](r)$
does not depend on the choice of $f$ which can be for instance the structural morphism of 
$X$ (indeed, since $Y$ is smooth, this is up to a shift a morphism of overconvergent isocrystal on $Y$ and 
then we can suppose $S$ smooth ; then this is a consequence
of the transitivity of the isomorphisms of the form
\ref{dual-smooth-morph}). 

For any $\K, \K' \in (F \text{-})D ^\mathrm{b} _\mathrm{hol}  (S/K)$,
using \ref{trans-epsilon} and \ref{II.11firstdiag},
we check the commutativity of the following diagram
\begin{equation}
\label{def1-theta_u-diag1}
\xymatrix{
{u ^{+} f ^{+} ( \K ' \otimes \K)} 
\ar[rr] ^-{\sim}
\ar[d] ^-{\sim} _-{\theta _g}
\ar@/_2cm/[dd] _-{\theta _u}
&& 
{u ^{+} f ^{+} ( \K ')  \otimes u ^{+} f ^{+} (  \K)} 
\ar[d] ^-{\sim} _-{\theta _g \otimes \mathrm{id}}
\ar@/^2.5cm/[dd] ^-{\theta _u \otimes \mathrm{id}}
\\
{u ^{!} f ^{!}  ( \K ' \otimes \K)  [-2 d _g](d _g)} 
&
{u ^{!} ( f ^{!}  ( \K ' )\otimes f ^{+} (  \K))  [-2 d _g](d _g)} 
\ar[l] ^-{\epsilon _f} _-{\sim}
& 
{u ^{!}  f ^{!} ( \K ')  \otimes u ^{+} f ^{+} (  \K)  [-2 d _g](d _g)} 
\ar[l] ^-{\epsilon _u} _-{\sim}
\ar@/_0.5cm/[ll] _-{\epsilon _g} ^-{\sim}
\\ 
{u ^{!} f ^{+}  ( \K ' \otimes \K)  [-2r](r)} 
\ar[u] _-{\sim} ^-{\theta _f}
\ar[r] ^-{\sim}
& 
{u ^{!} (f ^{+}  ( \K ') \otimes f ^{+}  ( \K))  [-2r](r)} 
\ar[u] _-{\sim} ^-{\theta _f \otimes \mathrm{id}}
& 
{u ^{!} f ^{+} ( \K ')  \otimes u ^{+} f ^{+} (  \K)  [-2r](r),} 
\ar[u] _-{\sim} ^-{\theta _f \otimes \mathrm{id}}
\ar[l] ^-{\epsilon _u} _-{\sim}
}
\end{equation}
where $\epsilon _f$, $\epsilon _g$ (and then $\epsilon _u$) are some isomorphisms because of the commutativity of the diagram
\ref{II.11firstdiag}.

More generally (in the context of the first part of the proposition), we define 
$\theta _u \colon u ^{+} (\E) \to u ^{!} (\E) [2r](r)$
so that the diagram 
\begin{equation}
\label{def1-theta_u}
\xymatrix{
{u ^{+} ( \E)} 
\ar[r] ^-{\sim}
\ar@{.>}[d] ^-{\theta _u }
&
{u ^{+} ( K _X \otimes \E)} 
\ar[r] ^-{\sim}
& 
{u ^{+} ( K _X )\otimes u ^{+} (\E)} 
\ar[d] ^-{\sim} _-{\theta _u \otimes id}
\\
{u ^{!} (\E) [2r](r)} 
&  
{u ^{!} ( K _X \otimes \E) [2r](r)} 
\ar[l] ^-{\sim}
& 
{u ^{!} ( K _X )\otimes u ^{+} (\E) [2r](r),} 
\ar[l] ^-{\epsilon _u}
}
\end{equation}
where $\theta _u \colon u ^{+} (K _X) \riso u ^{!} (K _X) [2r](r)$ is defined in \ref{predef1-theta_u} with $f$ equal to the structural 
morphism of $X$, 
is commutative.

We go back to the second part of the proposition, i.e. suppose now that $\E = f ^{+} (\K)$. 
By using the commutativity of the diagram \ref{def1-theta_u-diag1} applied to the case $\K' := K _S$, 
the isomorphism
$ \theta _u \colon u ^{+} (\E) \riso u ^{!} (\E) [2r](r)$
defined in \ref{predef1-theta_u}
is equal to that defined in \ref{def1-theta_u}. Hence, 
$\theta _u $ is indeed an isomorphism in this case. 


 \end{proof}

\begin{empt}
With the notation of \ref{KW-II.11.2-part1}, we get the morphism 
$u _{!}(\theta _u )\colon u _{!} u ^{+} (\E) \to u _{!}  u ^{!} (\E) [2r](r)$.
By adjunction, we have 
$u _{!}  u ^{!} (\E) \to \E$, which gives by composition
$\phi _u := \mathrm{adj} \circ u _{!}(\theta _u ) \colon 
u _{!} u ^{+} (\E) \to \E [2r](r)$.
The goal of this paragraph is to check  that the diagram below
\begin{equation}
\label{phi-phiotimes}
\xymatrix{
{u _{!}  u ^{+} (\E) } 
\ar[r] ^-{\mathrm{proj}} _-{\sim}
\ar[d] ^-{\phi _u }
& 
{u _{!}  u ^{+} (K _X) \otimes \E} 
\ar[d] ^-{\phi _u \otimes \mathrm{id}}
\\ 
{\E [2r](r)} 
\ar[r] ^-{\sim}
& 
{K _X \otimes  \E [2r](r).} 
}
\end{equation}
is commutative. 
It is sufficient to check the commutativity of the diagram
\begin{equation}
\notag
\xymatrix{
{u _{!}  u ^{+} (\E) }
\ar[r] ^-{\sim}
\ar[d] ^-{u _{!} (\theta _u)}
\ar@/_1,2cm/[dd] _-{\phi _u}
&
{u _{!}  u ^{+} (K _X \otimes \E) }
\ar[r] ^-{\sim}
\ar[d] ^-{u _{!} (\theta _u)}
& 
{u _{!}  (u ^{+} (K _X) \otimes  u ^{+} (\E) )}
\ar[r] ^-{\mathrm{proj}} _-{\sim}
\ar[d] ^-{u _{!} (\theta _u \otimes \mathrm{id})}
& 
{u _{!}  u ^{+} (K _X) \otimes \E} 
\ar[d] ^-{u _{!} (\theta _u) \otimes \mathrm{id}}
\ar@/^1,9cm/[dd] ^-{\phi _u \otimes \mathrm{id}}
\\ 
{u _{!}  u ^{!} (\E) [2r](r) }
\ar[r] ^-{\sim}
\ar[d] ^-{\mathrm{adj}}
&
{u _{!}  u ^{!} (K _X \otimes \E) [2r](r) }
\ar[d] ^-{\mathrm{adj}}
& 
{u _{!}  (u ^{!} (K _X) \otimes  u ^{+} (\E) )[2r](r) }
\ar[l] ^-{\epsilon _u}
\ar[r] ^-{\mathrm{proj}} _-{\sim}
& 
{u _{!}  u ^{!} (K _X) \otimes \E [2r](r) } 
\ar[d] ^-{\mathrm{adj}}
\\
{\E [2r](r)}
\ar[r] ^-{\sim}
&
{K _X \otimes \E [2r](r)}
\ar@{=}[rr] 
&& 
{K _X \otimes \E [2r](r).}
}
\end{equation}
From \ref{def1-theta_u}, the middle upper square is commutative.
The commutativity of the other squares is checked by functoriality. 
It remains to check the commutativity of the rectangle, which comes from 
the commutativity of the diagram below:
\begin{equation}
\notag
\xymatrix{
{u _{!}  u ^{!} (K _X \otimes \E) [2r](r) }
\ar@{=}[d] 
&
& 
{u _{!}  (u ^{!} (K _X) \otimes  u ^{+} (\E) )[2r](r) }
\ar[ll] ^-{\epsilon _u}
\ar[d] ^-{\mathrm{adj}}
\ar@{=}@/^2,5cm/[dd] ^-{}
\\ 
{u _{!}  u ^{!} (K _X \otimes \E) [2r](r) }
\ar[d] ^-{\mathrm{adj}}
&
{u _{!}  u ^{!} (u _{!}  u ^{!} (K _X )\otimes \E) [2r](r) }
\ar[l] ^-{\mathrm{adj}}
\ar[d] ^-{\mathrm{adj}}
& 
{u _{!}  u ^{!}  u _{!}  (u ^{!} (K _X) \otimes  u ^{+} (\E) )[2r](r) }
\ar[l] ^-{\sim} _-{\mathrm{proj}}
\ar[d] ^-{\mathrm{adj}}
\\ 
{K _X \otimes \E  [2r](r)} 
& 
{u _{!}  u ^{!} (K _X) \otimes  u ^{+} (\E) )[2r](r)} 
\ar[l] ^-{\mathrm{adj}}
& 
{u _{!}  (u ^{!} (K _X) \otimes  u ^{+} (\E) )[2r](r).} 
\ar[l] ^-{\sim} _-{\mathrm{proj}}
}
\end{equation}
We end this paragraph with a remark:
from the commutativity of the diagram \ref{phi-phiotimes}, 
we can construct the morphism $\phi _u \colon u _! u ^{+} (\E) \to \E [2r](r)$ and then 
 by adjunction 
$\theta _u \colon u ^{+} (\E) \to u ^{!} (\E) [2r](r)$
from 
$\phi _u \colon u _{!} u ^{+} (K _X) \riso K _X [2r](r)$.
%

\end{empt}

\begin{empt}
\label{empt-def-eta}
Let $u \colon Z \hookrightarrow X$ be a closed immersion of pure codimension $r$ in $X$ of smooth realizable $k$-varieties.
Let $\E \in (F \text{-})D ^\mathrm{b} _\mathrm{hol}  (X/K)$. 
From \ref{KW-II.11.2}, we have the morphism
$\theta _u \colon u ^{+} \E \to  u ^{!} \E [2r] (r)$.
The composition of the following three morphisms: 
\begin{equation}
\label{def-eta}
\eta _{u, \E} \colon 
\E
\overset{adj}{\longrightarrow}
u _{+} u ^{+} \E
\riso 
u _{!} u ^{+} \E
\overset{u _! (\theta _u)}{\longrightarrow}
u _{!} u ^{!} \E [2r] (r) 
\overset{adj}{\longrightarrow}
\E [2r] (r)
\end{equation}
is an element of 
$\mathrm{Hom} _{D ^\mathrm{b} _\mathrm{hol}  (X/K)}
\left ( \E , \E [2r]  \right ) $
(resp. $\mathrm{Hom} _{F \text{-}D ^\mathrm{b} _\mathrm{hol}  (X/K)}
\left ( \E , \E [2r] (r) \right ) $).
By using the commutativity of the diagram \ref{phi-phiotimes}, 
we check that
\begin{equation}
\label{empt-def-eta-otimes}
\eta _{u, \E} = \eta _{u, K _{X}} \otimes id _{\E} .
\end{equation}

\end{empt}

\begin{rem}
\label{rem-link3.2.6}
We keep the notation of \ref{empt-def-eta}.
Following the notation of \cite[3.1.1 and 3.1.6]{Abe-LanglandsIsoc}, we put
$H ^{2r} _{Z} (X) (r)
:= 
\mathrm{Hom} _{D ^\mathrm{b} _\mathrm{hol}  (X/K)}
( K _X , u _{+} u ^{!} K _{X} [2r] (r) )$
and 
$H ^{2r} (X) (r)
:= 
\mathrm{Hom} _{D ^\mathrm{b} _\mathrm{hol}  (X/K)}
( K _X , K _{X} [2r] (r) )$. 
From  \cite[3.1.6]{Abe-LanglandsIsoc}, 
the composition 
$u _+ (\theta _u) \circ adj \colon K _X 
\to  
u _{+} u ^{!} K _{X} [2r] (r) $ 
is called the cycle class of $Y$
and is denoted by $cl _{X} (Z) \in H ^{2r} _{Z} (X) (r)$.
Since $u _!$ is a left adjoint functor of $u ^!$ and since $u _+ \riso u _!$, 
we get a canonical homomorphism
$H ^{2r} _{Z} (X) (r) \to H ^{2r} (X) (r)$
which sends 
$cl _{X} (Z)$ to 
$\eta _{u, K _{X}}$ (see \ref{def-eta}).

\end{rem}

In order to check the theorem \ref{IV.1.3} below we will need 
the following lemmas:

\begin{lemm}
\label{inv-eta}
Let 
\begin{equation}
\notag
\xymatrix @ R=0,3cm {
{Z} 
\ar@{^{(}-)}[r] ^-{u}
&
{X} 
\\ 
{Z '} 
\ar@{^{(}-)}[r] ^-{u'}
\ar[u] ^-{g}
& 
{X'} 
\ar[u] ^-{f}
}
\end{equation}
be a cartesian square so that 
$u$ and $u'$ are closed immersions of pure codimension $r$ 
of smooth realizable $k$-varieties. 
Let $\E _X \in (F \text{-})D ^\mathrm{b} _\mathrm{hol}  (X/K)$ and 
$\E _{X'} := f ^{+}(\E _{X})$.
Let 
$\eta _{u, \E _X}\colon  \E _X  \to  \E _X [2r] (r)$
and
$\eta _{u ', \E _{X'}} \colon  \E _{X'}  \to  \E _{X'} [2r] (r)$
be the morphisms as defined in \ref{def-eta}.
Then we get $f ^{+} (\eta _{u , \E _{X}}) = \eta _{u ', \E _{X'}}$. 
\end{lemm}

\begin{proof}
Thanks to \ref{empt-def-eta-otimes}, we can suppose $\E _{X} =K _{X}$. This comes from \cite[3.2.6]{Abe-LanglandsIsoc} 
(see also the remark \ref{rem-link3.2.6}).
\end{proof}

\begin{lemm}
\label{pre0-eta-smooth-proj}
Let $\pi \colon \P ^{d} \to \Spec k $ be the canonical projection.
Let $\E \in (F \text{-})D ^\mathrm{b} _\mathrm{hol}  (\P ^{d}/K)$.
Let $H$ be the zero set of a section of the fundamental line bundle $\O _{\P ^{d} } (1)$
and $u \colon H \hookrightarrow \P ^{d}$ be the closed immersion.
The morphism 
$\eta _{u,\E} \colon  \E \to \E  [2] (1)$
as defined in
\ref{def-eta} does not depend on the choice of the hyperplane $H$
and will be denoted by $\eta _{\pi,\E}$. 
\end{lemm}

\begin{proof}
We can suppose $\E =K _{X}$.
Let $H _1, H _2$ be respectively the zero set of two sections of $\O _{\P ^{d}} (1)$.
From \ref{def-eta}, for $i =1,2$,
the closed immersions $u _i \colon H _{i} \hookrightarrow  \P ^{d} $
induce the morphisms 
$\eta _i \colon K _{X}\to K _{X}  [2] (1)$.
For $i=1,2$, 
we put
$\psi _i \colon K
\overset{adj}{\longrightarrow}
\pi _{+} \pi ^{+} K
=
\pi _{+} K _X
\overset{\pi _+ (\eta _i )}{\longrightarrow}
\pi _{+} K _X [2] (1)$.
By adjunction, 
$\eta _1= \eta _2 $
if and only if 
$\psi _1= \psi _2 $.
There exists an isomorphism $\sigma \colon \P ^{d}   \riso \P ^{d}  $ so that 
$\sigma ^{-1} (H _{1})= H _{2}$. 
From Lemma \ref{inv-eta}, 
we get 
$\sigma ^{+} (\eta _1) = \eta _2$
and then 
$\sigma _{+} (\eta _2) = \eta _1$.
Since $\pi \circ \sigma = \pi$, this implies that 
$\psi _2 = \psi _1$.
\end{proof}

\begin{empt}
\label{pre-eta-smooth-proj}
Let $S$ be a realizable variety, 
$\pi  \colon \P ^{d} \to \Spec k$ and
$\pi _S \colon \P ^{d} _S \to S$,
$f \colon \P ^{d} _S \to\P ^{d} $ be the canonical projections.
Let $\E \in (F \text{-})D ^\mathrm{b} _\mathrm{hol}  (\P ^{d} _S/K)$.
With the notation \ref{pre0-eta-smooth-proj}, 
we put 
\begin{equation}
\label{def-pre-eta-smooth-proj}
\eta _{\pi _S,\E}:= f ^{+}(\eta _{\pi ,K _{\P ^{d}}}) \otimes Id _{\E}\colon 
\E \to  \E  [2] (1).
\end{equation}

Let $S' \to S $ be a morphism of realizable varieties and
$a \colon \P ^{d} _{S'} \to \P ^{d} _S$ be the induced morphism. 
Then, we remark that 
\begin{equation}
\label{eq-pre-eta-smooth-proj}
a ^{+}(\eta _{\pi _S,\E})
=
\eta _{\pi _{S'}, a ^{+}(\E)}.
\end{equation}

\end{empt}

\begin{lemm}
\label{eta-smooth-proj}
Let $\pi \colon \P ^{d} _S \to S$ be the canonical projection and  
$\iota \colon \P ^{d'} _S \hookrightarrow  \P ^{d} _S$ be a closed $S$-immersion such that $\pi ': = \pi \circ \iota$
is the canonical projection.
Let $\E \in (F \text{-})D ^\mathrm{b} _\mathrm{hol}  (\P ^{d} _S/K)$.
We have the equality 
\begin{equation}
\label{eta-smooth-proj-def}
\iota ^{+} (\eta _{\pi,\E})
=
\eta _{\pi ',\iota ^{+}(\E)}.
\end{equation}
\end{lemm}

\begin{proof}
By construction (see \ref{def-pre-eta-smooth-proj}) 
we can suppose $\E =K _{\P ^{d}_S}$.
By using the property
\ref{eq-pre-eta-smooth-proj}, we reduce to treat the case $S= \Spec k$.
Then, this comes from \ref{inv-eta}.
\end{proof}

\begin{lemm}
\label{H0q+const}
Let $S$ be a realizable variety, 
$q\colon X=\A ^{d} _S \to S $ be the canonical projection. 
Let $\E \in \mathrm{Hol} (S/K)$.
\begin{enumerate}
\item For any $i \not = 0$, we have 
$\mathcal{H} ^{i} _t q _{+}  q ^{+} (\E)= 0$ and 
$\mathcal{H} ^{2d - i} _t q _{!}  q ^{+} (\E)= 0$. 

\item We have $\mathcal{H} ^{0} _t q _{+} q ^{+} (\E) \riso \E $
and 
$\mathcal{H} ^{2d} _t q _{!}  q ^{+} (\E)\riso \E$ 
in $\mathrm{Hol} (S/K)$.
\end{enumerate}

\end{lemm}

\begin{proof}
From \ref{dual-smooth-morph}, we can only consider the pushforward case. 
By transitivity of the pushforward, 
we reduce to the case where $d=1$. 
The complex $q _{+}  q ^{+} (\E)$ is isomorphic to the relative de Rham cohomology of $\A ^{1} _S /S$ of 
$q ^{+} [d] (\E)\in \mathrm{Hol} (\A ^{1} _S/K)$.
Then, this is an easy computation.
\end{proof}

\begin{thm}
\label{IV.1.3}
Let $\pi \colon \P ^{n} _S \to S$ be the canonical projection,  
$\iota \colon X \hookrightarrow  \P ^{n} _S$ be a closed immersion 
such that, for any closed point $s$ of $S$, 
$f ^{-1} (s) \riso \P ^{d} _{k(s)}$ where $k(s)$ is the residue field of $s$ and $f: = \pi \circ \iota$ (we might call such a morphism 
$f$ a $\P ^{d}$-fibration morphism). 
Let $\E \in (F \text{-})D ^\mathrm{b} _\mathrm{hol}  (S/K)$.
With the notation of \ref{def-pre-eta-smooth-proj}, we put
$$\eta = \iota ^{+}\eta _{\pi, \pi ^{+}(\E) [-2] (-1)}
\colon 
f ^{+}(\E) [-2] (-1) \to f ^{+}(\E).$$
By composition, for any integer $i \geq 0$, we get 
$\eta ^{i} \colon f ^{+}(\E) [-2i] (-i) \to f ^{+}(\E) $.
By adjunction, this is equivalent to have a morphism of the form
$\E [-2i] (-i) \to f _{+} \circ f ^{+}(\E) $, which by abuse of notation will still be denoted by 
$\eta ^{i}$.
The following map
\begin{equation}
\label{KW-IV.1.3}
\oplus _{i=0} ^{d} \eta ^{i}
\colon 
\bigoplus _{i=0} ^{d}
 \E [-2i] (-i) \to f _{+} \circ f ^{+}(\E) 
\end{equation}
is an isomorphism.
\end{thm}

\begin{proof}
The diagram
\begin{equation}
\label{1.15.2-diag}
\xymatrix{
{K _S \otimes \E} 
\ar[dd] ^-{\sim}
\ar[r] ^-{\mathrm{adj}\otimes Id _\E}
& 
{f _{+} f ^{+} (K _S)\otimes \E} 
\ar@{=}[r]
& 
{f _{+} (K _X)\otimes \E} 
\ar[rr] ^-{f _{+}(\eta)\otimes Id _\E}
\ar[d] ^-{\sim} _-{\mathrm{proj}}
&&
{f _{+} (K _X [2](1)) \otimes \E} 
\ar[d] ^-{\sim} _-{\mathrm{proj}}
\\
{} 
& 
{} 
& 
{f _{+} (K _X\otimes f ^{+}\E)} 
\ar[rr] ^-{f _{+}( \eta \otimes Id _{f ^{+}(\E)})}
\ar[d] ^-{\sim}
&&
{f _{+} (K _X [2](1) \otimes f ^{+}(\E))} 
\ar[d] ^-{\sim}
\\
{\E} 
\ar[rr] ^-{\mathrm{adj}}
& 
&
{f _{+} f ^{+} (\E)} 
\ar[rr] ^-{f _{+}(\eta _{f ^{+}(\E)})}
& &
{f _{+} f ^{+} (\E[2](1)),} 
}
\end{equation}
where the vertical arrows of the top are the projection isomorphisms (recall that 
since $f$ is proper, we have $f _+\riso f _!$), 
is commutative (indeed, the commutativity of the square below comes from the definition
\ref{def-pre-eta-smooth-proj}, that of the other square is functorial and that of the rectangle is left to the reader). 
By using the commutativity of the diagram \ref{1.15.2-diag}, we can suppose $\E= K _S$.

The fact that the morphism \ref{KW-IV.1.3} is actually an isomorphism can be checked after
pulling back by the closed immersions induced by the closed points of $S$. 
Hence, by using \ref{eq-pre-eta-smooth-proj}, 
we can suppose that $S= \Spec k$ and $X = \P ^{d}$.
From \ref{eta-smooth-proj}, 
we can suppose that $d =n$, i.e. $\iota $ is the identity and $f$ is the canonical projection
$\P ^{d} \to \Spec k $. 

We proceed by induction on $d \geq 0$.
The case $d= 0$ is obvious. 
So, we can suppose $d \geq 1$.
Let $q \colon \A ^{d}\to \Spec k$ the projection, 
$H:= X \setminus \A ^{d}$ be the hyperplane at the infinity, 
$u \colon H \hookrightarrow X$ the induced closed immersion,
$g := f \circ u $.
We put 
$ \widetilde{\eta} ^{i}:= u ^{+} (\eta ^{i}) 
\colon g ^{+}(K) [-2i] (-i) \to g ^{+}(K) $. Again by abuse of notation,
let $ \widetilde{\eta} ^{i}\colon K [-2i] (-i) \to g _{+}g ^{+}(K) $ be the morphism induced by adjunction.
From the transitivity of the adjunction morphism, we get the commutativity of the left square: 
\begin{equation}
\notag
\xymatrix @R=0,3cm{
{\eta ^{i}\colon}
&
{K [-2i] (-i) } 
\ar[r] ^-{\mathrm{adj}}
\ar@{=}[d] ^-{}
& 
{f _+ f ^{+} (K)[-2i] (-i) } 
\ar[r] ^-{f _{+}(\eta ^i)}
\ar[d] ^-{\mathrm{adj}}
& 
{f _+ f ^{+} (K)} 
\ar[d] ^-{\mathrm{adj}}
\\
{\widetilde{\eta} ^{i}   \colon}
&
{K[-2i] (-i) } 
\ar[r] ^-{\mathrm{adj}}
& 
{f _+ u _{+} u ^{+}  f ^{+} (K)[-2i] (-i) } 
\ar[r] ^-{g _{+}(\widetilde{\eta} ^{i}) }
& 
{f _+ u _{+} u ^{+} f ^{+} (K).} 
}
\end{equation}
This induces the following commutative square
\begin{equation}
\label{KW-IV.1.3-diag2}
\xymatrix @R=0,3cm{
{\oplus _{i=0} ^{d-1} \eta ^{i} \colon}
&
{\bigoplus _{i=0} ^{d-1} K [-2i] (-i) } 
\ar[r] ^-{}
\ar@{=}[d] ^-{}
& 
{f _+ f ^{+} (K)} 
\ar[d] ^-{\mathrm{adj}}
\\
{\oplus _{i=0} ^{d-1}\widetilde{\eta} ^{i}   \colon}
&
{\bigoplus _{i=0} ^{d-1} K[-2i] (-i) } 
\ar[r] ^-{\sim}
& 
{g _{+} g ^{+} (K).} 
}
\end{equation}

 a) From \ref{eta-smooth-proj-def}, we get the equality
$\widetilde{\eta}  = \eta _{g, g ^{+} (K)}$. 
By using the induction hypothesis applied to $g$, the arrow of the bottom of the diagram \ref{KW-IV.1.3-diag2} is an isomorphism.
We denote by $\tau _{\leq 2 d -1}$
the truncation functor of the canonical t-structure of \cite{Abe-Caro-weights}.
Since we have the exact triangle of localization
$q _{!} q ^{+} (K ) \to f _+ f ^{+} (K) \to g _{+} g ^{+} (K) \to +1$ and the Lemma \ref{H0q+const},
then after having applied the functor
$\tau _{\leq 2 d -1}$ to the right
morphism of \ref{KW-IV.1.3-diag2} 
we get an isomorphism (for the degree $2d -1$, we use that
$\mathcal{H} _t ^{2d-1}g _+ g ^+ (K)= 0$).
By considering \ref{KW-IV.1.3-diag2}, 
this implies that the truncation 
$\tau _{\leq 2 d -1} (\oplus _{i=0} ^{d} \eta ^{i})$
of \ref{KW-IV.1.3} is an isomorphism.

b) Now, consider the following commutative diagram:
\begin{equation}
\label{KW-IV.1.3-diag3}
\xymatrix @R=0,3cm{
{K [-2d] (-d) } 
\ar[r] ^-{\mathrm{adj}}
\ar@{=}[d] ^-{}
& 
{f _+ f ^{+} (K)[-2d] (-d) } 
\ar[r] ^-{f _+(\eta ^{d-1})}
\ar[d] ^-{\mathrm{adj}}
& 
{f _+ f ^{+} (K)  [-2](-1)} 
\ar[d] ^-{\mathrm{adj}}
\ar[r] ^-{f _{+}(\eta)}
&
{f _+ f ^{+} (K)} 
\\
{K[-2d] (-d) }
\ar[r] ^-{\mathrm{adj}}
\ar@/_0.8cm/[rr] ^-{\widetilde{\eta} ^{d-1}}
& 
{f _+ u _{+} u ^{+}  f ^{+} (K)[-2d] (-d) } 
\ar[r] ^-{g _{+}(\widetilde{\eta} ^{d-1})}
& 
{f _+ u _{+} u ^{+} f ^{+} (K) [-2](-1)} 
\ar[r] ^-{\sim} _-{g _{+}(\theta _u)}
& 
{f _+ u _{+} u ^{!} f ^{+} (K),} 
\ar[u] ^-{\mathrm{adj}}
}
\end{equation}
where the right arrow of the bottom is an isomorphism because of \ref{KW-II.11.2}.\ref{KW-II.11.2.2}
(the commutativity of the right square follows from the definition \ref{def-eta} and the notation \ref{def-pre-eta-smooth-proj}).
Hence, by using the induction hypothesis, we check that after having applied  the functor 
$\tau _{\geq 2 d}$ 
to the diagram
\ref{KW-IV.1.3-diag3},
the composition of the arrows of the bottom becomes an isomorphism.
A cone of the right morphism of \ref{KW-IV.1.3-diag3}
is isomorphic to $q _{+} q ^{+} (K)$.
From \ref{H0q+const}, we get
$ \tau _{\geq 2 d-1} q _{+} q ^{+} (K)=0$.
Hence, by applying
$\tau _{\geq 2 d}$ 
to the right morphism of \ref{KW-IV.1.3-diag3}, we get 
an isomorphism. 
This implies that $\tau _{\geq 2 d} (\eta ^{d})$
is an isomorphism.
Hence, so is 
$\tau _{\geq 2 d } (\oplus _{i=0} ^{d} \eta ^{i})$.
Using the step a) of the proof, we can conclude. 
\end{proof}

\begin{coro}
\label{III.7.7}
We keep the geometrical notation of \ref{IV.1.3} and we suppose $f$ smooth.
Let 
$\E 
\in 
F \text{-}D ^\mathrm{b} _\mathrm{hol}  (S/K) ^{\leq 0}$
and
$\FF \in
F \text{-}D ^\mathrm{b} _\mathrm{hol}  (S/K) ^{\geq 0}$.
Let $\G 
\in 
D ^\mathrm{b} _\mathrm{hol}  (S/K) ^{\leq 0}$
and
$\H \in
D ^\mathrm{b} _\mathrm{hol}  (S/K) ^{\geq 0}$.
Then 
\begin{gather}
\notag
Hom _{D ^\mathrm{b} _\mathrm{hol}  (X/K)}
(f ^{+} (\G), f ^{+}(\H))
=
Hom _{D ^\mathrm{b} _\mathrm{hol}  (S/K)}
(\G, \H);
\\
\label{III.7.7-7.8,11.2-eq1}
Hom _{F \text{-}D ^\mathrm{b} _\mathrm{hol}  (X/K)}
(f ^{+} (\E), f ^{+}(\FF))
=
Hom _{F \text{-}D ^\mathrm{b} _\mathrm{hol}  (S/K)}
(\E, \FF).
\end{gather}
\end{coro}

\begin{proof}
Since 
$\E \in 
F \text{-}D ^\mathrm{b} _\mathrm{hol}  (S/K) ^{\leq 0}$
and $\FF \in
F \text{-}D ^\mathrm{b} _\mathrm{hol}  (S/K) ^{\geq 0}$,
then 
\begin{equation}
Hom _{F \text{-}D ^\mathrm{b} _\mathrm{hol}  (S/K)}
(\E, \FF)
=
Hom _{F \text{-}D ^\mathrm{b} _\mathrm{hol}  (S/K)}
(\mathcal{H} ^{0} _t \E, \mathcal{H} ^{0} _t \FF).
\end{equation}
Since $f$ is smooth, 
then the functor 
$f _{+}f ^{+} $
preserves 
$F \text{-}D ^\mathrm{b} _\mathrm{hol}  (S/K) ^{\geq 0}$.
Hence, by adjunction, we get
\begin{equation}
\notag
Hom _{F \text{-}D ^\mathrm{b} _\mathrm{hol}  (X/K)}
(f ^{+} (\E), f ^{+}(\FF))
=
Hom _{F \text{-}D ^\mathrm{b} _\mathrm{hol}  (S/K)}
(\E, f _{+}f ^{+} (\FF))
=
Hom _{F \text{-}D ^\mathrm{b} _\mathrm{hol}  (S/K)}
(\mathcal{H} ^{0} _t  \E, \mathcal{H} ^{0} _t  f _{+}f ^{+} (\FF)).
\end{equation}
With $\mathcal{H} ^{0} _t \FF
\underset{\ref{IV.1.3}}{\riso}
\mathcal{H} ^{0} _t f _{+} \circ f ^{+} (\mathcal{H} ^{0} _t \FF)
\riso 
\mathcal{H} ^{0} _t f _{+} \circ f ^{+} (\FF)$,
then we obtain the last equality of \ref{III.7.7-7.8,11.2-eq1}. 
The proof without Frobenius is identical.
\end{proof}

\begin{prop}
\label{III.7.8}
We keep the notation and hypotheses of 
\ref{III.7.7}.
\begin{enumerate}
\item The functor 
$f ^{+} [d]\colon (F \text{-})\mathrm{Hol} (S/K) \to 
 (F \text{-})\mathrm{Hol} (X/K)$
is t-exact and fully faithful.

\item For any 
$\E, \FF \in \mathrm{Hol} (S/K)$, the functor 
$f ^{+}[d]$ induces  the equality
\begin{equation}
\label{Ext1-const}
\mathrm{Ext} ^{1}   _{ \mathrm{Hol} (S/K)}
(\E,  \FF)
=
\mathrm{Ext}  ^{1} _{ \mathrm{Hol} (X/K)}
(f ^{+}[d](\E), f ^{+}[d](\FF)).
\end{equation}
\end{enumerate}
\end{prop}

\begin{proof}
Since $f$ is smooth, 
the functor $f ^{+} [d]$ is t-exact. 
From \ref{III.7.7}, we get its full faithfulness. 
Since the canonical morphism
$\FF [1] \to  \tau _{\leq 0 }f _{+}f ^{+} (\FF [1])$
is an isomorphism (use \ref{KW-IV.1.3}), 
we get the second assertion (by using similar technics than in the proof of \ref{III.7.7}).
\end{proof}

\begin{rem}
\label{rmkFrob-const}
We keep the notation and hypotheses of \ref{III.7.7}.
Let $\E \in \mathrm{Hol} (S/K)$.
Since the pull back under Frobenius commutes with the functor 
$f ^{+} [d]$, 
then we get a bijection between Frobenius structures on $\E$ and 
Frobenius structures on $f ^{+} [d] (\E)$. 
Moreover, let $\FF , \G \in F \text{-}\mathrm{Hol} (S/K)$ and
$\phi \colon \FF \to \G$ a morphism of $\mathrm{Hol} (S/K)$.
Then $\phi$ commutes with Frobenius if and only if so is 
$f ^{+} [d] (\phi)$.
\end{rem}

\begin{prop}
\label{11.2}
We keep the notation and hypotheses of 
\ref{III.7.7}.
We suppose furthermore that the morphism $f$ has locally a section. 
\begin{enumerate}
\item The functor 
$f ^{+} [d]\colon (F \text{-}) \mathrm{Hol} (S/K) 
\to 
 (F \text{-})\mathrm{Hol} (X/K)$ sends simple objects to simple objects. 
 
\item The functor $f ^{+} [d]\colon (F \text{-}) \mathrm{Hol} (S/K) 
\to 
 (F \text{-})\mathrm{Hol} (X/K)$ has the right adjoint functor
 $\mathcal{H} ^{-d} _t \circ f _+
  \colon
 (F \text{-})\mathrm{Hol} (X/K) \to 
  (F \text{-})\mathrm{Hol} (S/K)$
and the left adjoint functor
$\mathcal{H} ^{d} _t \circ f _! (d) 
\colon
(F \text{-})\mathrm{Hol} (X/K) \to 
(F \text{-})\mathrm{Hol} (S/K)$.  

\end{enumerate}
\end{prop}

\begin{proof}
Let $\E$ be a simple object of $(F \text{-})\mathrm{Hol} (S/K)$.
From \cite[1.4.9.(i)]{Abe-Caro-weights} (without Frobenius structures, use the fact that 
$\mathrm{Isoc} ^{\dag \dag} (S/K)\cap \mathrm{Hol} (S/K)$ is a Serre subcategory of $\mathrm{Hol} (S/K)$, 
where $\mathrm{Isoc} ^{\dag \dag} (S/K)$ is constructed in \cite{caro-pleine-fidelite})
there exist an open dense smooth subscheme $S'$ of $S$, 
an irreducible object $\E' \in (F \text{-})\mathrm{Hol} (S'/K)$ 
which is also an object of $(F\text{-})\mathrm{Isoc} ^{\dag \dag} (S'/K)$
 such that 
$\E \riso u _{!+} (\E')$ where $u \colon S' \hookrightarrow S$ is the inclusion. 
Put $X ':= f ^{-1} (S')$, 
$f ' \colon X ' \to S'$, 
$v \colon X' \hookrightarrow X$.
By adjunction, we remark that the canonical morphism 
$u _{!} (\E') \to u _{+} (\E')$ is the only one so that 
we get the identity over $S'$.
With this remark,
since $f ^{+} [d] \circ u _{!} (\E') \riso v _{!} \circ f ^{\prime +} [d]  (\E') $, 
and $f ^{+} [d] \circ u _{+} (\E') 
\riso 
f ^{!} [-d](-d) \circ u _{+} (\E') 
\riso 
v _{+} \circ f ^{\prime !} [-d](-d)  (\E') 
\riso
v _{+} \circ f ^{\prime +} [d] (\E') $, 
since
the functors $f ^{ +} [d]$ and $f ^{\prime +} [d]$ are t-exact, 
then after applying 
$\mathcal{H} ^{0} _t$ to these isomorphisms
we get 
 the isomorphism
$f ^{+} [d]  \circ u _{!+} (\E')
\riso
v _{!+}   \circ f ^{\prime +} [d] (\E')$.
Since the functor $v _{!+}  $ preserves the irreducibility, 
then we reduce to the case where $S$ is affine, smooth, irreducible, and
where we have moreover $\E \in (F\text{-})\mathrm{Isoc} ^{\dag \dag} (S/K)$.
Hence 
$f ^{ +} [d] (\E) \in  (F\text{-})\mathrm{Isoc} ^{\dag \dag} (X/K)\cap (F \text{-})\mathrm{Hol} (X/K)$.
Let $0 \not =\G \in (F \text{-})\mathrm{Hol} (X/K) $ be a subobject of 
$f ^{ +} [d] (\E)$.
Since 
$(F\text{-})\mathrm{Isoc} ^{\dag \dag} (X/K)\cap (F \text{-})\mathrm{Hol} (X/K)$ is a Serre subcategory of 
$(F\text{-})\mathrm{Hol} (X/K)$, then
$\G  \in  (F\text{-})\mathrm{Isoc} ^{\dag \dag} (X/K)$. 
Since the generic rank of an overconvergent isocrystal is preserved under pull-backs, 
since $f$ has locally a section, since $X$ is irreducible (because the fibers of $f$ are irreducible) then we can conclude that
$\G= f ^{ +} [d] (\E) $ and hence 
$f ^{ +} [d] (\E) $ is a simple object.

The last part comes from the left t-exactness of $f _+ [-d]$
and the right t-exactness of $f _! [d](d)$, from the fact that the couples
$(f ^{+} [d], f _+ [-d])$ 
and 
$( f _! [d](d), f ^{!} [-d] (-d))$ 
are adjoint functors
and from the isomorphism
$f ^{+} [d] \riso f ^{!} [-d] (-d)$ of \ref{dual-smooth-morph}.
\end{proof}

\begin{prop}
\label{ThmIII.11.3-4}
We keep the notation and hypotheses of 
\ref{11.2}.
Let $\E \in  (F \text{-})\mathrm{Hol} (X/K) $.
\begin{enumerate}
\item The category of constant objects with respect to $f$ is a thick subcategory of 
$ (F \text{-})\mathrm{Hol} (X/K) $. 

\item 
The object
$\mathcal{H} ^{0} _t \circ f ^{+} \circ f _{+} (\E)=
f ^{+} [d]\circ (\mathcal{H} ^{-d} _t  f _{+} )(\E)$
is the largest constant with respect to $f$ subobject of $\E$ in
$(F \text{-})\mathrm{Hol} (X/K) $.

\item The object
$\mathcal{H} ^{2d} _t \circ f ^{+} \circ f _{!} (\E(d))=
f ^{+} [d]\circ (\mathcal{H} ^{d} _t  f _{!} )(\E(d))$
is the largest constant with respect to $f$ quotient object of $\E$ in
$(F \text{-})\mathrm{Hol} (X/K) $.

\end{enumerate}

\end{prop}

\begin{proof}
The thickness of the category of constant objects without Frobenius structures comes from the equality \ref{Ext1-const}.
With the remark \ref{rmkFrob-const}, we get the thickness with Frobenius structures.
The rest is similar to the proof of \cite[III.11.3]{KW-Weilconjecture} i.e. this
comes from the general fact \cite[III.11.1]{KW-Weilconjecture} and from \ref{11.2}.
\end{proof}

\section{The Brylinski-Radon transform and the Hard Lefschetz Theorem}
Let $d \geq 1$ be an integer and 
$\P ^d $ be the $d$-dimensional projective space defined over $k$, 
let $\check{\P} ^{d}$ be the dual projective space over $k$, which parameterizes the hyperplanes in $\P ^d$, 
let $H$ be the universal incidence relation, i.e. the closed subvariety of $\P ^d \times \check{\P} ^{d}$ so that $(x,h) \in H$ 
if and only if the point $x \in h$. Let $Y$ be a realizable $k$-variety. 
We denote by $i \colon H \times Y  \hookrightarrow \P ^d \times \check{\P} ^{d} \times Y$ the canonical immersion
and
$p _1 \colon \P ^d \times \check{\P} ^{d} \times Y \to \P ^d  \times Y$, 
$p _2 \colon \P ^d \times \check{\P} ^{d} \times Y \to \check{\P} ^{d} \times Y$, 
$\widetilde{p} _1 \colon \check{\P}  ^d \times Y \to  Y$, 
$\widetilde{p} _2 \colon \P ^d \times Y \to Y$
the canonical projections and 
$\pi _1 := p _1 \circ i$,
$\pi _2 := p _2 \circ i$.

\begin{dfn}
\label{B-R-transf-def}
We define the Brylinski-Radon transform
$Rad \colon F \text{-}D ^\mathrm{b} _\mathrm{hol}  (\P ^{d} \times Y/K)
\to 
F \text{-}D ^\mathrm{b} _\mathrm{hol}  (\check{\P} ^{d} \times Y/K)$ 
by posing, for any $\E \in F \text{-}D ^\mathrm{b} _\mathrm{hol}  (\P ^{d} \times Y/K)$,
\begin{equation}
\label{B-R-transf-Rad}
Rad (\E) := \pi _{2+} \pi _1 ^{+} (\E) [d-1].
\end{equation}
For any $n \in \Z$, we put 
$Rad ^{n}:= \mathcal{H} _t ^{n} \circ Rad$.

\end{dfn}

\begin{dfn}
Let $U$ be the open complement of the closed subvariety $H \times Y$ in 
$\P ^{d}  \times \check{\P}  ^{d} \times Y$. Let $j \colon U \hookrightarrow \P ^{d}  \times \check{\P}  ^{d} \times Y$ be the open immersion
and $q _1:= p _1 \circ  j $, $q _2:= p _2 \circ  j$.
We define the modified Radon transform 
$Rad _!\colon 
F \text{-}D ^\mathrm{b} _\mathrm{hol}  (\P ^{d} \times Y/K)
\to 
F \text{-}D ^\mathrm{b} _\mathrm{hol}  (\check{\P} ^{d} \times Y/K)$ 
by posing,
for any $\E \in F \text{-}D ^\mathrm{b} _\mathrm{hol}  (\P ^{d} \times Y/K)$,
\begin{equation}
Rad _! (\E) 
:= q _{2!} \circ q _{1} ^{+} [d] (\E).
\end{equation}
For any integer $n \in \Z$, 
we put
$Rad ^{n }_{!} (\E)  := \mathcal{H} _t ^{n} Rad _{!} (\E) $.

\end{dfn}

\begin{empt}
\label{Rad!properties}
\begin{enumerate}
\item \label{Rad!properties1} Since the functor $ q _{1} ^{+} [d]$ is t-exact and 
the functor $q _{2!}$ is left t-exact (because $q _2$ is affine, e.g. see \cite[1.3.13]{Abe-Caro-weights} but this is obvious here
since $q _2$ is moreover smooth)
then $Rad _!$ is left t-exact.

\item The exact triangle 
$i _{+} \circ i ^{+}[-1]\to j _{!} j ^{+} \to id \to i _{+} \circ i ^{+}$
induces for any $\E \in F \text{-}D ^\mathrm{b} _\mathrm{hol}  (\P ^{d} \times Y/K)$ the exact triangle
\begin{equation}
\label{Rad-Rad!}
Rad (\E) \to Rad _{!} (\E) \to p _{2 +} p _{1} ^{+} [d] (\E) 
\to Rad (\E) [1].
\end{equation}

\end{enumerate}

\end{empt}

\begin{lem}
\label{IV.2.6}
Let $\E \in F \text{-}D ^\mathrm{b} _\mathrm{hol}  (\P ^{d} \times Y/K)$.
\begin{enumerate}
\item We have the isomorphism $\E [-2d](-d) \riso q _{1!}  q _1^{+} (\E) $.
\item We have the isomorphism $\widetilde{p} _{1+} (Rad  _{!}(\E)) \riso \widetilde{p} _{2+} (\E) [-d] (-d)$.
\item \label{IV.2.6-3} If  $\E \in  (F \text{-})D ^{\geq 0}  (\P ^{d} \times Y/K)$, then
$\widetilde{p} _{1+}(Rad  _{!}(\E)) \in 
(F \text{-})D ^{\geq 0}  (Y/K)$.
\end{enumerate}
\end{lem}

\begin{proof}
We put $\eta = \eta _{p _1, p _1 ^{+}(\E) [-2] (-1)}$
and $ \widetilde{\eta} ^{i}:= i ^{+} (\eta ^{i}) $.
In order to check the first isomorphism, 
by using the remark of the paragraph \ref{def-K_X} and the projection formula (see  \cite[A.6]{Abe-Caro-weights})
we get $q _{1!}  q _1^{+} (\E) \riso q _{1!} ( q _1^{+} (K _U) \otimes q _1^{+} (\E) ) 
\riso q _{1!}  q _1^{+} (K _U) \otimes \E$.
Hence, we can suppose $\E = K _{\P ^{d} \times Y}$.
Then, by using a base change theorem, we can suppose $Y = \Spec k$.
We put $\G = K _{\P ^{d}} [-d]\in F \text{-}\mathrm{Hol}  (\P ^{d} /K)$.
Consider the diagram below
\begin{equation}
\label{IV.2.6-diag2}
\xymatrix {
{q _{1!}  q ^{+} _1 (\G) } 
\ar[r] ^-{}
& 
{p _{1+} p _{1} ^{+}(\G)} 
\ar[r] ^-{\mathrm{adj}}
& 
{ p _{1+} i _{+} i ^{+}  p _{1} ^{+}(\G)} 
\ar[r] ^-{+}
&
{}
\\
{\G [-2d] (-d) } 
\ar[r] ^-{}
& 
{\bigoplus _{i=0} ^{d} \G [-2i] (-i) } 
\ar[r] ^-{}
\ar[u] _-{\oplus _{i=0} ^{d} \eta ^{i}} ^-{\sim}
& 
{\bigoplus _{i=0} ^{d-1} \G [-2i] (-i) } 
 \ar[u] _-{\oplus _{i=0} ^{d-1} \widetilde{\eta} ^{i}} ^-{\sim}
 \ar[r] ^-{+}
&
{,}
}
\end{equation}
where the rows are  exact triangles and where we keep 
the abuse of notation in \ref{IV.1.3}, i.e.
the morphism
$\eta ^{i}$ (resp. $\widetilde{\eta} ^{i}$) 
means the morphism induced by adjunction with respect to the couple 
$(p _{1} ^+ , p _{1+})$ 
(resp. $(q _{1} ^+ , q _{1+})$) from $\eta ^{i}$ (resp. $\widetilde{\eta} ^{i}$).
By transitivity of the adjunction, we get that the square of \ref{IV.2.6-diag2} is commutative. 
Moreover, we recall that the vertical arrows are isomorphisms thanks to \ref{KW-IV.1.3}. 
By applying the functor $\tau _{\geq 2d}$ to the diagram \ref{IV.2.6-diag2}, 
we get 
$q _{1!}  q ^{+} _1 (\G)  \underset{\ref{H0q+const}}{\riso} 
\tau _{\geq 2d} q _{1!}  q ^{+} _1 (\G) 
\riso 
\tau _{\geq 2d}p _{1+} p _{1} ^{+}(\G)
\liso 
\tau _{\geq 2d}
\bigoplus _{i=0} ^{d} \G [-2i] (-i) 
\liso
\G [-2d] (-d) $, which finishes the proof of the first isomorphism.
We get the second isomorphism from the first one by composition:
$$\widetilde{p} _{1+}(Rad  _{!}(\E)) 
\riso
\widetilde{p} _{1!} \circ q _{2!} \circ q _{1} ^{+} [d] (\E)
\riso
\widetilde{p} _{2!} \circ q _{1!} \circ q _{1} ^{+} [d] (\E)
\riso 
\widetilde{p} _{2!} (\E )[-d](-d).$$
Finally, since $\widetilde{p} _{2!}[-d]$ is left t-exact, we obtain the third property from the second one.
\end{proof}

\begin{lem}
\label{IV.2.7}
Let 
$\E \in (F \text{-})D ^{\geq 0}  (\P ^{d} \times Y/K)$.
Then $ Rad ^{0} _{!}(\E) $ 
is left reduced with respect to $\widetilde{p} _{1}$, i.e. does not have any nontrivial constant with respect to $\widetilde{p} _{1}$
subobject. 
\end{lem}

\begin{proof}
The proof is the same than \cite[IV.2.7]{KW-Weilconjecture}: 
from Proposition \ref{ThmIII.11.3-4} we reduce to prove that 
$\mathcal{H} ^{-d} _t  \widetilde{p} _{1+} (Rad ^{0} _{!}(\E))=0$.
Since 
$Rad _{!}(\E) \in  (F \text{-})D ^{\geq 0}  (\check{\P} ^{d} \times Y/K)$ 
(see the property \ref{Rad!properties}.\ref{Rad!properties1}), 
since $\mathcal{H} ^{-d} _t  \widetilde{p} _{1+} $ is left t-exact, 
then we get the isomorphism
$\mathcal{H} ^{-d} _t  \widetilde{p} _{1+} (Rad ^{0} _{!}(\E)) 
\riso \mathcal{H} ^{-d} _t  \widetilde{p} _{1+} (Rad _{!}(\E))$. 
We conclude by using \ref{IV.2.6}.\ref{IV.2.6-3}

\end{proof}

\begin{empt}
\label{CherHypplane}
Let $f \colon X \to Y$ be a projective morphism of realizable varieties. 
Let $\E \in F \text{-}D ^\mathrm{b} _\mathrm{hol}  (X/K)$. 
A morphism $\eta \colon \E\to \E [2](1)$ is called 
a ``Chern class of a relative hyperplane for the projective morphism $f$''
if there exists a closed immersion
$\iota \colon X \hookrightarrow  \P ^{d} _Y$ 
so that
$f = \pi \circ \iota$, 
where 
$\pi \colon \P ^{d} _Y \to Y$ is the canonical projection,
and so that, with the notation \ref{pre-eta-smooth-proj}, 
$$\eta = id _{\E} \otimes  \iota ^{+} (\eta _{\pi, K _{\P ^{d} _Y}}).$$
By using the projection isomorphisms (see  \cite[A.6]{Abe-Caro-weights}),
we remark that 
$\iota _{+} (\eta) \colon \iota _{+} (\E ) \to  \iota _{+} (\E ) [2](1)$
is canonically isomorphic to  $\mathrm{id} _{\iota _{+} (\E )} \otimes \eta _{\pi, K _{\P ^{d} _Y}}$.

\end{empt}

\begin{theo}
[Hard Leftschetz Theorem]
\label{HLT-IV.4.1}
Let $f \colon X \to Y$ be a projective morphism of realizable varieties. 
Let 
$\E \in F \text{-}\mathrm{Hol}  (X/K)$ be an $\iota$-pure module
and 
$\eta  \colon \E \to \E [2] (1)$
be a Chern class of a relative hyperplane for $f$ (see the definition of \ref{CherHypplane}).
For any positive integer $r$, 
we obtain by composition 
$\eta ^{r}\colon \E \to \E [2r] (r)$.
We get the homomorphism
\begin{equation}
\label{HLT-IV.4.1eta-r}
\mathcal{H} _t ^{-r} f _+ (\eta ^{r}) \colon \mathcal{H} _t ^{-r} f _{+} (\E ) 
\to 
\mathcal{H} _t ^{r} f _{+} (\E ) (r).
\end{equation}
The homomorphism \ref{HLT-IV.4.1eta-r} is an isomorphism.
\end{theo}

\begin{proof}
We follow the proof of the Hard Leftschetz Theorem of 
 \cite[IV.4.1]{KW-Weilconjecture} which is similar 
 to that of \cite{BBD}.

0. Since the assertion is local on $Y$, one can suppose $Y$ affine and smooth. 
Using the remark of \ref{CherHypplane}, we reduce to the case where $f$ is the projection
$ \P ^{d} _Y \to Y$
and
$\eta = id _{\E} \otimes \eta _{f, K _{\P ^{d} _Y}}$. Then we keep the notation of the section.

1. In this step, we treat the case $r=1$. 
We put $\G = p _{1}  ^{+} [d] (\E)\in F \text{-}\mathrm{Hol}  (\P ^{d}  \times \check{\P}  ^{d} \times Y/K)$. 
Following \ref{def-eta}, 
we get from the closed immersion $i \colon H \times Y \hookrightarrow 
\P ^{d}  \times \check{\P}  ^{d} \times Y$ 
the morphism
$\zeta :=\eta _{i, \G} \colon \G \to \G [2](1)$.

a) 
Let $\Spec k  \hookrightarrow   
\check{\P}  ^{d}$ be a rational section,
$t \colon 
Y
\hookrightarrow
\check{\P}  ^{d} \times Y$
and 
$s \colon 
\P ^{d}  \times Y
\hookrightarrow
\P ^{d}  \times \check{\P}  ^{d} \times Y$
the induced closed immersions. 
Since $s ^{-1} (H \times Y)$ is an hyperplane of  
$\P ^{d}\times Y$, since $s$ is a section of $p _1$, 
using Lemma \ref{inv-eta}
we get 
$s ^{+}[-d](\zeta) = \eta$. 
Since $ p _{2+} p _{1} ^{+} \riso \widetilde{p} _{1} ^{+} \widetilde{p} _{2 +}$
and $ \widetilde{p} _{2+} s ^{+} \riso t ^{+}p _{2 +}$,
since the functor $s ^{+}[-d]$
(resp. $t ^{+}[-d]$) is acyclique for the constant objects with respect to $ p _{1}$
(resp. $\widetilde{p} _1$), we get 
that 
$$t ^{+} [-d]\mathcal{H} _t ^{-1} p _{2+} (\zeta )
\riso 
\mathcal{H} _t ^{-1}  \widetilde{p} _{2+}  s ^{+} [-d](\zeta )
=
\mathcal{H} _t ^{-1}  \widetilde{p} _{2+}  (\eta).$$
Hence, this is enough to check that 
$\mathcal{H} _t ^{-1} p _{2+} (\zeta )
\colon 
\mathcal{H} _t ^{-1} p _{2+} 
(\G)
\to 
\mathcal{H} _t ^{1} p _{2+} (\G)(1)$
is an isomorphism. 
For simplicity, we denote this morphism $\zeta$.

b) Consider the diagram
\begin{equation}
\label{IV.3.2}
\xymatrix @R=0,3cm @C=0,3cm{
{ \mathcal{H} _t ^{d-1} p _{2 +} p _{1} ^{+} (\E) }
\ar@{=}[r] ^-{}
\ar[d] ^-{}
&
{\mathcal{H} _t ^{-1} p _{2+}(\G)} 
\ar[r] ^-{\zeta}
\ar[d] ^-{\mathrm{adj}}
&
{\mathcal{H} _t ^{1} p _{2+} (\G)(1)}
\ar[r] ^-{\sim}
&
{\DD\mathcal{H} _t ^{-1} p _{2+} (\check{\G})(1)} 
\ar@{=}[r] ^-{}
&
{\DD \mathcal{H} _t ^{d-1} p _{2 +} p _{1} ^{+} (\DD (\E)) (1)}
\\
{Rad ^{0} (\E)}
&
{\mathcal{H} _t ^{-1} p _{2+}i _+ i ^{+} (\G)} 
\ar[r] ^-{\sim} _-{\theta _i}
\ar[l] ^-{\sim} 
& 
{\mathcal{H} _t ^{1} p _{2+} i _+ i ^{!} (\G)(1)}  
\ar[u] ^-{\mathrm{adj}}
\ar[r] ^-{\sim}
&
{\DD \mathcal{H} _t ^{-1} p _{2+} i _+ i ^{+} (\check{\G})(1)}  
\ar[u] ^-{\DD (\mathrm{adj})}
&
{\DD Rad ^{0} (\DD (\E))(1),}  
\ar[u] ^-{}
\ar[l] ^-{\sim}
}
\end{equation}
where 
$\check{\G}= p _{1}  ^{+} [d] (\DD (\E)) \in F \text{-}\mathrm{Hol}  (\P ^{d}  \times \check{\P}  ^{d} \times Y/K)$,
the horizontal isomorphisms of the middle square are constructed using the commutativity of the dual functor 
with the functor $p _{1}  ^{+} [d] $ (this is the Poincare duality \ref{dual-smooth-morph}) and with proper push-forwards, 
where the right square is the dual of the left square used for $\DD (\E)$ instead of $\E$,
is commutative. By transitivity of the relative duality isomorphism and by definition of the adjunction morphisms, 
we get the commutativity of the middle square. The commutativity of the other squares of \ref{IV.3.2} are tautological 
(e.g. for the second square, this is the construction \ref{def-eta}). 
Moreover, the second left arrow of the bottom row is indeed an isomorphism because of \ref{KW-II.11.2}.\ref{KW-II.11.2.2}.

c) Since $\pi _{1}$ is smooth of relative dimension $d-1$ and
$\pi _{2}$ is proper, 
we get 
$\DD Rad ^{0} (\DD (\E))
\riso 
Rad ^{0} (\E)$.
From both properties of \ref{Rad!properties},
we get the first exact sequence
\begin{gather}
\label{exseq1}
0\to \mathcal{H} _t ^{d-1} p _{2 +} p _{1} ^{+} (\E)  \to Rad ^{0} (\E) \to Rad ^{0 }_{!} (\E) ,
\\  
\label{exseq2}
\DD Rad ^{0 }_{!} (\DD \E) 
\to 
Rad ^{0} (\E)
\to 
\DD 
\mathcal{H} _t ^{d-1} p _{2 +} p _{1} ^{+} (\DD \E)  
\to 
0,
\end{gather}
the second one is induced by duality. 
By construction, the first morphism of \ref{exseq1} and the last one of \ref{exseq2} are respectively the left vertical arrow and
the right vertical arrow of \ref{IV.3.2}.
Using Lemma \ref{IV.2.7} we check that 
$ Rad ^{0} _{!}(\E) $ 
(resp. $\DD Rad ^{0 }_{!} (\DD \E)$) 
is left (resp. right) reduced with respect to $\widetilde{p} _{1}$, i.e. does not have any nontrivial constant with respect to $\widetilde{p} _{1}$
subobject (resp. quotient). This implies that 
$\mathcal{H} _t ^{d-1} p _{2 +} p _{1} ^{+} (\E) $ 
(resp. $\DD \mathcal{H} _t ^{d-1} p _{2 +} p _{1} ^{+} (\DD \E) $)
is the maximal constant with respect to $\widetilde{p} _{1}$ subobject (resp. quotient) of
$Rad ^{0} (\E) $. 

d) Since $\pi _{1}$ is smooth,
$\pi _{2}$ is proper and $\E$ is $\iota$-pure then so is 
$Rad ^{0} (\E) $. Hence $Rad ^{0} (\E) $ is semi-simple in the category 
$\mathrm{Hol}  (\check{\P} ^{d} \times Y/K)$ (see \cite[4.3.1]{Abe-Caro-weights}).
By considering the diagram \ref{IV.3.2} and using the step 1.c), this implies that
the morphism 
$\zeta
\colon \mathcal{H} _t ^{-1} p _{2+}(\G) 
\to 
\mathcal{H} _t ^{1} p _{2+} (\G)(1)$
is an isomorphism in 
$\mathrm{Hol}  (\check{\P} ^{d} \times Y/K)$.
Since $\zeta$ is also a morphism of 
$F \text{-}\mathrm{Hol}  (\check{\P} ^{d} \times Y/K)$, 
then $\zeta$ is  an isomorphism of $F \text{-}\mathrm{Hol}  (\check{\P} ^{d} \times Y/K)$.

2. We proceed by induction on $r$. Suppose $r \geq 2$. 
We put 
$\widetilde{\G} := i ^{+} (\G) [-1] \riso \pi _{1} ^+ [d-1] (\E)$.
The morphism $\zeta$ induces by pull-back
$\widetilde{\zeta} ^{r-1}:= i ^{+} (\zeta ^{r-1})  [-1] 
\colon 
\widetilde{\G}
\to 
\widetilde{\G}
[2r](r)$. 
Consider the commutative diagram:
\begin{equation}
\label{HLT-IV.4.1}
\xymatrix @R=0,3cm @C=0,3cm{
{\mathcal{H} _t ^{-r} p _{2+}(\G)} 
\ar[r] ^-{\zeta ^{r-1}}
\ar[d] ^-{\mathrm{adj}}
&
{\mathcal{H} _t ^{r-2} p _{2+}(\G)} 
\ar[r] ^-{\zeta}
\ar[d] ^-{\mathrm{adj}}
&
{\mathcal{H} _t ^{r} p _{2+} (\G)(1)}
\ar[r] ^-{\sim}
&
{\DD\mathcal{H} _t ^{-r} p _{2+} (\check{\G})(1)} 
\\
{\mathcal{H} _t ^{-r} p _{2+}i _+ i ^{+} (\G)} 
\ar[r] ^-{\zeta ^{r-1}}
\ar[d] ^-{\sim}
&
{\mathcal{H} _t ^{r-2} p _{2+}i _+ i ^{+} (\G)} 
\ar[r] ^-{\sim} _-{\theta _i}
\ar[d] ^-{\sim} 
& 
{\mathcal{H} _t ^{r} p _{2+} i _+ i ^{!} (\G)(1)}  
\ar[u] ^-{\mathrm{adj}}
\ar[r] ^-{\sim}
&
{\DD \mathcal{H} _t ^{-r} p _{2+} i _+ i ^{+} (\check{\G})(1)}  
\ar[u] ^-{\DD (\mathrm{adj})}
\\
{\mathcal{H} _t ^{-(r-1)} \pi _{2+} (\widetilde{\G})} 
\ar[r] ^-{\widetilde{\zeta} ^{r-1}}
&
{\mathcal{H} _t ^{r-1} \pi _{2+}(\widetilde{\G}),} 
}
\end{equation}
where the middle arrow of the middle row is an isomorphism because of \ref{KW-II.11.2}.\ref{KW-II.11.2.2}.
By considering the long exact sequence induced 
by the exact triangle \ref{Rad-Rad!}, 
since $Rad _!$ is left exact and $r \geq 2$, 
we check that the 
adjunction morphism
$\mathcal{H} _t ^{-r} p _{2+}(\G) \to
\mathcal{H} _t ^{-r} p _{2+}i _+ i ^{+} (\G) $
is an isomorphism.
This implies that the right vertical arrow of \ref{HLT-IV.4.1} is an isomorphism.
To check that the arrow of the bottom of the diagram \ref{HLT-IV.4.1} is an isomorphism, 
it is sufficient to prove that 
$t ^{+} [-d] \mathcal{H} _t ^{-(r-1)} \pi _{2+}(\widetilde{\zeta} ^{r-1})$ is an isomorphism (the rational point $\Spec k  \hookrightarrow   \check{\P}  ^{d}$
can vary).

We put $\widetilde{i} \colon s ^{-1} (H \times Y) \hookrightarrow \P ^{d}\times Y$,
$\widetilde{s}\colon s ^{-1} (H \times Y) \hookrightarrow H \times Y$,
$\widetilde{\pi} _2:= \widetilde{p} _2 \circ \widetilde{s}$.
From Lemma \ref{inv-eta}, since $s ^{+}[-d](\G) \riso \E$, we have
$s ^{+} [-d] (\zeta)= \eta _{\widetilde{i}, \E}$. 
Hence, since $\widetilde{i}$ is a closed immersion induced by an hyperplane of $\P ^{d}\times Y$, 
putting $\widetilde{\E}:= \widetilde{i} ^{+} [-1] (\E)\in F \text{-}\mathrm{Hol}  (s ^{-1}(H)/K)$, 
we remark that $\widetilde{\eta} := \widetilde{i} ^{+} [-1] s ^{+}  [-d] (\zeta)\colon 
\widetilde{\E} \to \widetilde{\E} [2](1)$ is 
a Chern class of a relative hyperplane for the projective morphism $\widetilde{\pi} _2$.
Hence, since $\widetilde{\E}$ is pure, 
by using the induction hypothesis, 
we get that $ \mathcal{H} _t ^{-(r-1)}  \widetilde{\pi} _{2+}  (\widetilde{\eta} ^{r-1})$ is an isomorphism.
Finally,
with the same arguments than in the step 1.a), we check the first isomorphism:
\begin{gather}
\notag
t ^{+} [-d]\mathcal{H} _t ^{-(r-1)} \pi _{2+} (\widetilde{\zeta} ^{r-1})
\riso 
\mathcal{H} _t ^{-(r-1)}  \widetilde{\pi} _{2+}  \widetilde{s} ^{+} [-d](\widetilde{\zeta} ^{r-1})
=
\mathcal{H} _t ^{-(r-1)}  \widetilde{\pi} _{2+}  \widetilde{s} ^{+} [-d]
 i ^{+}  [-1]  (\zeta ^{r-1})
 \\
 \notag
 \riso 
 \mathcal{H} _t ^{-(r-1)}  \widetilde{\pi} _{2+}  \widetilde{i} ^{+} [-1]
 s ^{+}  [-d]  (\zeta ^{r-1})
 \riso
 \mathcal{H} _t ^{-(r-1)}  \widetilde{\pi} _{2+}  (\widetilde{\eta} ^{r-1}),
\end{gather}
which implies that $t ^{+} [-d]\mathcal{H} _t ^{-(r-1)} \pi _{2+} (\widetilde{\zeta} ^{r-1})$ is also an isomorphism.
\end{proof}

\section{The dual Brylinski-Radon and the inversion formula}

We keep the notation of the chapter $2$.

\begin{dfn}
\label{B-R-transf-def}
We define the dual Brylinski-Radon transform
$Rad ^{\vee} \colon F \text{-}D ^\mathrm{b} _\mathrm{hol}  (\check{\P} ^{d} \times Y/K)
\to 
F \text{-}D ^\mathrm{b} _\mathrm{hol}  (\P ^{d} \times Y/K)$ 
by posing, for any $\check{\E} \in F \text{-}D ^\mathrm{b} _\mathrm{hol}  (\check{\P} ^{d} \times Y/K)$,
\begin{equation}
\label{B-R-transf-Rad-dual}
Rad ^{\vee} (\check{\E}) := \pi _{1+} \pi _2 ^{+} (\check{\E}) [d-1].
\end{equation}
\end{dfn}

\begin{lemm}
\label{lem-IV.1.4}
Let $\iota \colon X := ( H \times _{\check{\P} ^{d}} \check{H} )  \times Y\hookrightarrow \P ^{d} \times \check{\P}  ^{d} \times \P ^{d} \times Y$ 
be the canonical embedding,
$p _{13}\colon \P ^{d} \times \check{\P}  ^{d} \times \P ^{d} \times Y \to \P ^{d} \times \P ^{d} \times Y $ be the projection
and $\pi= p _{13} \circ \iota $.
Let $\Delta \colon \P ^{d} \times Y  \hookrightarrow \P ^{d} \times \P ^{d} \times Y$
be the diagonal immersion (and the identity over $Y$).
Let
$ \FF \in F \text{-}D ^\mathrm{b} _\mathrm{hol}  (\P ^{d}  \times \P ^{d} \times Y/K)$.
We have the isomorphism of $F \text{-}D ^\mathrm{b} _\mathrm{hol}  (\P ^{d}  \times \P ^{d} \times Y/K)$ of the form
\begin{equation}
\notag
\Delta _{+} \circ \Delta ^{+} (\FF) [2-2d] ( 1-d)
\bigoplus
\oplus _{i=0} ^{d-2} \FF [-2i](-i)
\riso
\pi _{+} \pi ^{+} (\FF).
\end{equation}

\end{lemm}

\begin{proof}
By using the projection isomorphisms (see  \cite[A.6]{Abe-Caro-weights}),
we can suppose $\FF = K _{\P ^{d} \times \P ^{d} \times Y}$.
Hence, by using some base change theorems (induced by the projection
$Y \to \Spec k$), 
we can suppose $Y = \Spec k$.
We put $\widetilde{K} :=K _{\P ^{d} \times \P ^{d}}$.
Consider the following cartesian squares
\begin{equation}
\xymatrix{
{X} 
\ar@{^{(}->}[r] ^-{\iota}
\ar@{}[dr]|{\square}
\ar@/^0,6cm/[rr] _-{\pi}
& 
{ \P ^{d} \times \check{\P}  ^{d} \times \P ^{d}} 
\ar[r] ^-{p _{13}}
\ar@{}[dr]|{\square}
& 
{ \P ^{d}\times \P ^{d}} 
\\ 
{\widetilde{X}} 
\ar@{^{(}->}[r] ^-{\widetilde{\iota}}
\ar@{^{(}->}[u] ^-{\widetilde{\Delta}}
\ar@/_0,6cm/[rr] ^-{\widetilde{\pi}}
& 
{\P ^{d} \times \check{\P}  ^{d}} 
\ar[r] ^-{p _1}
\ar@{^{(}->}[u] ^-{}
& 
{\P ^{d}.}
\ar@{^{(}->}[u] ^-{\Delta}
}
\end{equation}
Put $\widetilde{\pi}:= p _{1} \circ \widetilde{\iota}$,
 $\eta := \iota ^{+}\eta _{p _{13}, K _{\P ^{d} \times \check{\P}  ^{d}  \times \P ^{d}}}$
and 
 $\widetilde{\eta} := \widetilde{\iota} ^{+}\eta _{p _{1}, K _{\P ^{d} \times \check{\P}  ^{d}  }}$.
From the formula \ref{eq-pre-eta-smooth-proj},
we check 
$\widetilde{\Delta} ^+(\eta )
=
\widetilde{\eta} $.
By using the construction of Theorem \ref{IV.1.3}, we get the morphism
\begin{equation}
\label{lem-IV.1.4-ar1}
\oplus _{i=0} ^{d-2}
\eta ^{i}
\colon
\oplus _{i=0} ^{d-2} \widetilde{K} [-2i](-i)
\to
\pi _{+} \pi ^{+} (\widetilde{K}).
\end{equation}
Since $\pi$ is outside $\Delta (\P ^{d} )$ a $\P ^{d-2}$-fibration, 
from Theorem \ref{IV.1.3}, 
the morphism \ref{lem-IV.1.4-ar1} is an isomorphism outside $\Delta (\P ^{d} )$. 
Hence, a cone of \ref{lem-IV.1.4-ar1} is in the essential image of $\Delta _+$.
Since $\Delta ^{+} \pi _{+}\pi ^{+} (\widetilde{K}) \riso \widetilde{\pi} _{+} \widetilde{\Delta} ^{+} \pi ^{+} (\widetilde{K})
\riso 
 \widetilde{\pi} _{+}  \widetilde{\pi} ^{+} \Delta ^{+}  (\widetilde{K}) $, 
by applying $\Delta ^{+} $ 
to the morphism 
\ref{lem-IV.1.4-ar1}, 
we get 
\begin{equation}
\label{lem-IV.1.4-ar2}
\oplus _{i=0} ^{d-2}
\widetilde{\eta} ^{i}
\colon
\oplus _{i=0} ^{d-2} \Delta ^{+}  \widetilde{K} [-2i](-i)
\to
 \widetilde{\pi} _{+}  \widetilde{\pi} ^{+} \Delta ^{+}  (\widetilde{K}).
\end{equation}
Since $\widetilde{\pi}$ is a $\P ^{d-1}$-fibration, 
from 
Theorem \ref{IV.1.3}, 
we remark that the cone of the morphism
\ref{lem-IV.1.4-ar2} is isomorphic to 
$\Delta ^{+} (\widetilde{K}) [2-2d] ( 1-d)$.
Moreover, using Theorem \ref{IV.1.3} again
we build the morphism
$\Delta ^{+} \pi _{+}\pi ^{+} (\widetilde{K}) \to 
\Delta ^{+} (\widetilde{K}) [2-2d] ( 1-d)$ and then by adjunction
the second morphism of the sequence in $F \text{-}D ^\mathrm{b} _\mathrm{hol}  (\P ^{d}  \times \P ^{d} /K)$:
\begin{equation}
\label{lem-IV.1.4-ar2bis}
\oplus _{i=0} ^{d-2} \widetilde{K} [-2i](-i)
\overset{\ref{lem-IV.1.4-ar1}}{\longrightarrow} 
\pi _{+} \pi ^{+} (\widetilde{K})
\to
\Delta _{+} \circ \Delta ^{+} (\widetilde{K}) [2-2d] ( 1-d).
\end{equation}
Since $\pi _{+} \pi ^{+} (\widetilde{K})$ is $\iota$-pure, then 
$\pi _{+} \pi ^{+} (\widetilde{K})$ is semisimple 
in $D ^\mathrm{b} _\mathrm{hol}  (\P ^{d}  \times \P ^{d} /K)$
(see \cite[4.3.6]{Abe-Caro-weights}).
Then, we get the isomorphism
$\Delta _{+} \circ \Delta ^{+} (\widetilde{K}) [2-2d] ( 1-d)
\bigoplus
\oplus _{i=0} ^{d-2} \widetilde{K} [-2i](-i)
\riso
\pi _{+} \pi ^{+} (\widetilde{K})$
in $D ^\mathrm{b} _\mathrm{hol}  (\P ^{d}  \times \P ^{d} /K)$
which induces the morphisms of \ref{lem-IV.1.4-ar2bis}.
For any $i =0,\dots, d-2$, we have
\begin{gather}
\notag
\bullet \ \mathrm{Hom} _{D ^\mathrm{b} _\mathrm{hol}  (\P ^{d}  \times \P ^{d}/K)}
\left (\Delta _{+} \circ \Delta ^{+} (\widetilde{K}) [2-2d] ,  \widetilde{K} [-2i]  \right )
\riso
\mathrm{Hom} _{D ^\mathrm{b} _\mathrm{hol}  (\P ^{d}  /K)}
\left ( \Delta ^{+} (\widetilde{K}) [2-2d] , \Delta ^{!}  \widetilde{K} [-2i]  \right )
\\
\underset{\ref{KW-II.11.2-part1}}{\riso}
\mathrm{Hom} _{D ^\mathrm{b} _\mathrm{hol}  (\P ^{d}  /K)}
\left ( \Delta ^{!} (\widetilde{K}) [2] , \Delta ^{!}  \widetilde{K} [-2i]  \right )
=0.\\
\bullet \ \mathrm{Hom} _{D ^\mathrm{b} _\mathrm{hol}  (\P ^{d}  \times \P ^{d}/K)}
\left (\widetilde{K} [-2i]  , \Delta _{+} \circ \Delta ^{+} (\widetilde{K}) [2-2d] \right )
\riso 
\mathrm{Hom} _{D ^\mathrm{b} _\mathrm{hol}  (\P ^{d}  /K)}
\left (\Delta ^{+}  \widetilde{K} [-2i] , \Delta ^{+} (\widetilde{K}) [2-2d]\right ) =0.
\end{gather}
Hence, we get the compatibility with Frobenius.
\end{proof}

\begin{prop}
[Radon Inversion Formula]
\label{IV.1.4}
Let 
$\E \in F \text{-}D ^\mathrm{b} _\mathrm{hol}  (\P ^{d} \times Y/K)$.
Then the following formula holds
\begin{equation}
Rad ^{\vee}  \circ Rad (\E) 
\riso 
\E (1-d) \oplus \widetilde{p} _2 ^{+} [d] (\phi (\E)),
\end{equation}
where 
$\phi (\E)
:= 
\oplus _{i=0} ^{d-2} \widetilde{p} _{2+} (\E) [d-2-2i](-i)
$
\end{prop}

\begin{proof}
With the notation of \ref{lem-IV.1.4}, let respectively 
$u,v 
\colon 
\P ^{d} \times \P ^{d} \times Y \to \P ^{d} \times Y$
be the left and middle projection. 
Then, by using the base change theorem (more precisely, look at the cartesian square
defining the fibered product
$X = (H \times Y)\times _{\check{\P} ^{d}\times Y} (\check{H} \times Y) $)
we get
$Rad ^{\vee}  \circ Rad (\E) 
\riso 
v _{+} \pi _{+} \pi ^{+} u ^{+} (\E)  
[2d-2]$.
Hence we obtain: 
\begin{gather}
\notag
Rad ^{\vee}  \circ Rad (\E) 
\underset{\ref{lem-IV.1.4}}{\riso}
v _+ \Delta _{+} \circ \Delta ^{+} (u ^{+} (\E))( 1-d)
\bigoplus
\oplus _{i=0} ^{d-2}v _+ u ^{+} (\E) [2d -2i-2](-i)
\\
\riso
\E ( 1-d)
\bigoplus
\oplus _{i=0} ^{d-2}
\widetilde{p} _2 ^{+} [d] \widetilde{p} _{2+} (\E) [d -2i-2](-i).
\end{gather}

\end{proof}

\bibliographystyle{smfalpha}

\providecommand{\bysame}{\leavevmode ---\ }
\providecommand{\og}{``}
\providecommand{\fg}{''}
\providecommand{\smfandname}{et}
\providecommand{\smfedsname}{\'eds.}
\providecommand{\smfedname}{\'ed.}
\providecommand{\smfmastersthesisname}{M\'emoire}
\providecommand{\smfphdthesisname}{Th\`ese}

\bigskip
\noindent Daniel Caro\\
Laboratoire de Mathématiques Nicolas Oresme\\
Université de Caen
Campus 2\\
14032 Caen Cedex\\
France.\\
email: daniel.caro@unicaen.fr

\end{document}